\documentclass[a4paper, 12pt,oneside,reqno]{amsart}
\usepackage[a4paper]{geometry}
\usepackage{amssymb,amsmath,amsthm}
\usepackage[T2A]{fontenc}
\usepackage[breaklinks=true]{hyperref}

\usepackage{xcolor}

\usepackage{bbm}
\usepackage{graphicx}

\def\Var{\mathrm{Var}}
\def\dim{\mathrm{dim}}
\def\Der{\mathrm{Der}}
\def\Nov{\mathrm{Nov}}

\def\Lie{\mathrm{Lie}}

\def\pre{\mathrm{pre}}

\def\M{\mathrm{M}}
\def\Mock{\mathrm{Mock}}

\def\Der {\mathop {\fam 0 Der}\nolimits}

\let\<\langle
\let\>\rangle

\newtheorem{definition}{Definition}

\newtheorem{proposition}{Proposition}
\newtheorem{theorem}{Theorem}

\newtheorem{remark}{Remark}

\title{4-type subvarieties of the variety of associative algebras}

\author{A. Kunanbayev}

\address{Institute of Mathematics and Mathematical Modeling, Almaty, Kazakhstan}

\email{abai-aga@mail.ru}

\author{B. Sartayev$^{*}$}

\address{Narxoz University, Almaty, Kazakhstan}

\email{baurjai@gmail.com}



\thanks{${}^{*}$Corresponding author: Bauyrzhan Sartayev   (baurjai@gmail.com)}

\begin{document}

\maketitle

\begin{abstract}
In this paper, we consider four types of subvarieties of the variety of associative algebras. We study these subvarieties from the point of view of operads and show their connections with well-known classes of algebras, such as dendriform algebras and noncommutative Novikov algebras. Finally, we define the commutator and anti-commutator operations on these algebras and derive several identities satisfied by these operations.
\end{abstract}

\section{Introduction}

Associative algebras are closely related to several important classes of algebras, including Lie algebras, Jordan algebras, dendriform algebras, and noncommutative Novikov algebras. For example, the Poincaré-Birkhoff-Witt theorem provides the embedding of any Lie algebra into an appropriate associative algebra under the commutator. Some Jordan algebras can be obtained from associative algebras using the anti-commutator. However, there exist exceptional Jordan algebras that cannot be embedded into any associative algebra via the anti-commutator \cite{Cohn}. For more details and examples of finite-dimensional associative algebras, see \cite{As1, As2}.

Let us consider the variety of associative algebras equipped with operators such as derivations and Rota–Baxter operators. If we define new operations $\succ$ and $\prec$ on a differential associative algebra by
\begin{equation}\label{newop}
a\succ b=a'b\textrm{\;\;and\;\;}a\prec b=ab',
\end{equation}
then the resulting algebra satisfies the defining identities of the variety of noncommutative Novikov algebras \cite{DauSar, erlagol2021}, which are given by
\begin{gather}
    x\succ (y\prec z) = (x\succ y)\prec z,\label{eq:LIdent1} \\
    (x\prec y)\succ z - x\succ (y\succ z) = x\prec (y\succ z) - (x\prec y)\prec z. \label{eq:LIdent2}
\end{gather}
A more general result states that any noncommutative Novikov algebra can be embedded into an appropriate differential associative algebra.

Analogously, if we define new operations $\succ$ and $\prec$ on an associative algebra equipped with a Rota–Baxter operator $R$ by
\begin{equation}\label{newop2}
a\succeq b=R(a)b\textrm{\;\;and\;\;}a\preceq b=aR(b),
\end{equation}
then the resulting algebra satisfies the defining identities of the variety of dendriform algebras \cite{RB2, RB1}, which are given by
\begin{gather}
(a \preceq b) \preceq c = a \preceq (b \preceq c + b \succeq c), \label{eq:LIdent3} \\
(a \succeq b) \preceq c = a \succeq (b \preceq c), \label{eq:LIdent4} \\
(a \preceq b + a \succeq b) \succeq c = a \succeq (b \succeq c). \label{eq:LIdent5}
\end{gather}
Moreover, a free dendriform algebra can be embedded into a free associative algebra with a Rota–Baxter operator. 

It turns out that all these results can be explained in terms of the theory of operads, i.e., the white Manin product of the operad $\Nov$ with any quadratic operad $\Var$ gives an operad $\Der\Var$ which can be embedded into $\Var^{(d)}$ under the operations (\ref{newop}). Analogously, the black Manin product of the operad $\pre$-$\Lie$ with any quadratic operad $\Var$ gives an operad $\pre$-$\Var$, which can be embedded into $\Var^{(R)}$ under the operations (\ref{newop2}). As universal objects, these observations can be presented as follows: 

\begin{picture}(30,80)
\put(188,60){$\mathcal{V}ar^{(d,R)}$}
\put(138,30){$\mathcal{V}ar^{(d)}$}
\put(248,30){$\mathcal{V}ar^{(R)}$}
\put(168,47){\normalsize\rotatebox[origin=c]{45}{$\hookrightarrow$}}
\put(228,47){\normalsize\rotatebox[origin=c]{135}{$\hookrightarrow$}}
\put(58,0){$\mathcal{D}er\mathcal{V}ar:=\mathcal{N}ov\circ\mathcal{V}ar$}
\put(250,0){$pre\textrm{-}\mathcal{V}ar:=pre\textrm{-}\mathcal{L}ie\bullet\mathcal{V}ar$}
\put(115,15){\tiny{(\ref{newop})}\normalsize\rotatebox[origin=c]{45}{$\hookrightarrow$}}
\put(255,15){\tiny{(\ref{newop2})}\normalsize\rotatebox[origin=c]{135}{$\hookrightarrow$}}
\end{picture}

$ $

In \cite{MarklRemm}, it was shown that by decomposing the associative product as $ab=1/2([a,b]+\{a,b\})$, the associative identity becomes equivalent to the following two axioms:
\[
[\{a,c\},b]=\{[a,b],c\}-\{a,[b,c]\},
\]
\[
[b,[a,c]]=\{\{a,b\},c\}-\{a,\{b,c\}\}.
\]
This decomposition is known as polarization, a method of rewriting a non-symmetric binary operation as a combination of symmetric and antisymmetric operations.

In this paper, we define four types of associative algebras based on Mal'cev's classification. In \cite{Mal'cev}, it was shown that the space of identities of degree $3$ in associative algebras can be decomposed into subspaces defined by the following identities:
\begin{gather}
abc+bac+acb+cab+bca+cba=0,
\label{as1} \\
abc-bac-acb+cab+bca-cba=0,
\label{as2} \\
abc+bac-bca-cba=0,
\label{as3} \\
abc+acb-cba-cab=0.
\label{as4}
\end{gather}
Using these identities, we define the corresponding subvarieties of the variety of associative algebras. For these subvarieties, we define the classes of algebras based on the diagram. These subvarieties are closely related to the varieties of alternative, assosymmetric, and left-alternative algebras. Moreover, they may provide useful tools for addressing the problem of identifying Jordan elements within free associative algebras. Notably, when equipped with the commutator operation, these subvarieties yield some of the most well-known classes of Lie algebras, including metabelian Lie algebras, nilpotent Lie algebras, and Lie algebras satisfying an additional identity of degree $5$.

Indeed, every identity of degree $3$ in an associative algebra can be expressed as a linear combination of the listed identities. Moreover, each of these identities generates an invariant subspace within the space of multilinear polynomials of degree $3$ in the associative algebra.

One of the well-known subvarieties of the variety of associative algebras is the variety of perm algebras \cite{perm1, perm2, MS2022}. Analogical subvarieties for the variety of alternative algebras were considered in \cite{KazMat}. Since there is a one-to-one correspondence between a variety of algebras $\Var$ and the quadratic operad associated with it, we adopt the same terminology for both throughout this paper. We denote by $\Var\<X\>$ the free algebra in the variety of algebras $\Var$ generated by the set $X$. Also, we use the notion $\Var^{(d)}$ for the variety of algebras $\Var$ with derivation $d$.

We consider all algebras over a field $K$ of characteristic $0$.

\section{First-type associative algebras}

\begin{definition}
An algebra is called a first-type associative algebra if it satisfies the identity
$$abc+bac+acb+cab+bca+cba=0.$$
\end{definition}

We denote by $\mathcal{A}s_1$ the variety of first-type associative algebras. One motivation for studying algebras in $\mathcal{A}s_1$ is the following well-known result:
\begin{proposition}
\cite{Dzhuma} The dual operad of $\mathcal{A}s_1$ is the alternative operad.
\end{proposition}
It was also shown in \cite{Dzhuma} that the operad $\mathcal{A}s_1$ is not Koszul.

Let us now list the defining identities of the operad $\Der \mathcal{A}s_1:=\Nov \circ \mathcal{A}s_1$, where $\circ$ denotes the white Manin product of two quadratic operads. Direct computation gives a noncommutative Novikov operad with the following identities:
\begin{multline*}
(a\succ b)\succ c-a\succ (b\succ c)+(a\succ c)\succ b-a\succ (c\succ b)+(b\prec a)\succ c-b\succ (a\succ c)\\
+b\prec(c\prec a)-(b\prec c)\prec a+(c\prec a)\succ b-c\succ (a\succ b)+c\prec(b\prec a)-(c\prec b)\prec a=0
\end{multline*}
and
\[
a\succ(b\succ c)+(a\succ c)\prec b+b\succ(a\succ c)+(b\succ c)\prec a+(c\prec a)\prec b+(c\prec b)\prec a
=0.
\]

\begin{proposition}
Any algebra in the variety $\Der\mathcal{A}s_1$ can be embedded into an appropriate algebra from the variety $\mathcal{A}s_1^{(d)}$.
\end{proposition}
\begin{proof}
Since $\Nov\circ\mathcal{A}s_1=\Nov\otimes\mathcal{A}s_1$, the result follows directly from \cite{KolMashSar} and \cite{erlagol2021}.
\end{proof}

Next, consider the operad $\pre\text{-}\mathcal{A}s_1:=\pre\text{-}\Lie \bullet \mathcal{A}s_1$, where $\bullet$ denotes the black Manin product of two quadratic operads. The corresponding dendriform operad satisfies the identity:
\begin{multline*}
(a\preceq b)\preceq c+(b\succeq a)\preceq c+(a\preceq c)\preceq b+(c\succeq a)\preceq b\\
+(b\succeq c)\succeq a+(b\preceq c)\succeq a+(c\succeq b)\succeq a+(c\preceq b)\succeq a=0.
\end{multline*}

\begin{proposition}
Define a new operation $\star$ on the operad $\pre\text{-}\mathcal{A}s_1$ by
\[a\star b=a\preceq b+a\succeq b.\]

Then the algebra $(X, \star)$ is an $\mathcal{A}s_1$-algebra. In fact, the free algebra $\pre\text{-}\mathcal{A}s_1\<X\>$ can be embedded into the algebra $\mathcal{A}s_1\<X\>$ equipped with a Rota--Baxter operator $R$, as follows:
\[
a\preceq b=aR(b)\;\;\textrm{and}\;\;a\succeq b=R(a)b.
\]
\end{proposition}

\begin{proposition}\label{polarization}
The polarization of $\mathcal{A}s_1$ is given by the identities:
\[
[\{a,c\},b]=\{[a,b],c\}-\{a,[b,c]\},
\]
\[
[b,[a,c]]=\{\{a,b\},c\}-\{a,\{b,c\}\}
\]
and
\[
\{a,\{b,c\}\}+\{\{a,c\},b\}+\{\{a,b\},c\}=0.
\]
\end{proposition}
\begin{proof}
Firstly, let us rewrite each identity in $\mathcal{A}s_1$ by commutator and anti-commutator as follows:
\[
ab=1/2([a,b]+\{a,b\}).
\]
We define an order on monomials with operations $[\cdot,\cdot]$ and $\{\cdot,\cdot\}$ by
\begin{itemize}
    \item $[[\cdot,\cdot],\cdot]>[\{\cdot,\cdot\},\cdot]>\{[\cdot,\cdot],\cdot\}>\{\{\cdot,\cdot\},\cdot\}$;
    \item the remaining monomials are ordered by lexicographical order;
\end{itemize}
Let $\mathcal{R}$ be a set of identities in $\mathcal{A}s_1$ in terms of $[\cdot,\cdot]$ and $\{\cdot,\cdot\}$, i.e., we have
\begin{multline*}
(ab)c-a(bc)=1/4([[a,b],c]+[\{a,b\},c]+\{[a,b],c\}+\{\{a,b\},c\}\\
+[[b,c],a]-\{[b,c],a\}+[\{b,c\},a]-\{\{b,c\},a\}),    
\end{multline*}
\begin{multline*}
(ba)c-b(ac)=1/4(-[[a,b],c]+[\{a,b\},c]-\{[a,b],c\}+\{\{a,b\},c\}\\
+[[a,c],b]-\{[a,c],b\}+[\{a,c\},b]-\{\{a,c\},b\}),
\end{multline*}
\begin{multline*}
(ac)b-a(cb)=1/4([[a,c],b]+[\{a,c\},b]+\{[a,c],b\}+\{\{a,c\},b\}\\
-[[b,c],a]+\{[b,c],a\}+[\{b,c\},a]-\{\{b,c\},a\}),
\end{multline*}
\begin{multline*}
(ca)b-c(ab)=1/4(-[[a,c],b]+[\{a,c\},b]-\{[a,c],b\}+\{\{a,c\},b\}\\
+[[a,b],c]-\{[a,b],c\}+[\{a,b\},c]-\{\{a,b\},c\}),   
\end{multline*}
\begin{multline*}
(bc)a-b(ca)=1/4([[b,c],a]+[\{b,c\},a]+\{[b,c],a\}+\{\{b,c\},a\}\\
-[[a,c],b]-\{[a,c],b\}+[\{a,c\},b]-\{\{a,c\},b\}),   
\end{multline*}
\begin{multline*}
(cb)a-c(ba)=1/4(-[[b,c],a]+[\{b,c\},a]-\{[b,c],a\}+\{\{b,c\},a\}\\
-[[a,b],c]-\{[a,b],c\}+[\{a,b\},c]-\{\{a,b\},c\}),
\end{multline*}
\begin{multline*}
abc+bac+acb+cab+bca+cba=\\
1/2([\{a,b\},c]+\{\{a,b\},c\}
+[\{a,c\},b]+\{\{a,c\},b\}+[\{b,c\},a]+\{\{b,c\},a\}.
\end{multline*}
For the set $\mathcal{R}$, we construct a matrix $[\mathcal{R}]$ representing the identities as row vectors relative to the given order.
\[ \left(
\begin{array}{cccccccccccc}
1/4& 0& 1/4& 1/4& 0& 1/4& -1/4& 0& 1/4& -1/4& 0& 1/4\\
0& 1/4& -1/4& 0& 1/4& 1/4& 0& -1/4& -1/4& 0& -1/4& 1/4\\
-1/4& 1/4& 0& 1/4& 1/4& 0& 1/4& 1/4& 0& -1/4& 1/4& 0\\
0& -1/4& 1/4& 0& 1/4& 1/4& 0& -1/4& -1/4& 0& 1/4& -1/4\\
1/4& -1/4& 0& 1/4& 1/4& 0& 1/4& -1/4& 0& 1/4& -1/4& 0\\
-1/4& 0& -1/4& 1/4& 0& 1/4& -1/4& 0& -1/4& 1/4& 0& -1/4\\
0& 0& 0& 1/2& 1/2& 1/2& 0& 0& 0& 1/2& 1/2& 1/2
\end{array} \right)
\]

By applying Gröbner basis theory, the matrix $[\mathcal{R}]$ can be brought to an echelon form, and we obtain the result.
\end{proof}

\begin{theorem}
All identities of the algebra $\mathcal{A}s_1^{(-)}\<X\>$ up to degree $4$ follow from the anti-commutativity and Jacobi identities.
\end{theorem}
\begin{proof}
Let us consider the free Lie algebra denoted by $\Lie\<X\>$. The basis monomials of the multilinear component of $\Lie\<X\>$ in degrees 3 and 4 are
\[
[[a,b],c],\;[[a,c],b]
\]
and 
\[
[[[a,b],c],d],\;[[[a,b],d],c],\;[[[a,c],b],d],\;[[[a,c],d],b],\;[[[a,d],b],c],\;[[[a,d],c],b],
\]
respectively. To prove that there are no additional identities in the algebra $\mathcal{A}s_1^{(-)}\<X\>$ up to degree 4, it is sufficient to show that the following equations admit only the trivial solutions:
\[
\lambda_1[[a,b],c]+\lambda_2[[a,c],b]=0,
\]
\[
\beta_1[[[a,b],c],d]+\beta_2[[[a,b],d],c]+\beta_3[[[a,c],b],d]+\beta_4[[[a,c],d],b]+\beta_5[[[a,d],b],c]+\beta_6[[[a,d],c],b]=0.
\]
The calculations for these equations are straightforward.
\end{proof}

\begin{theorem}
An algebra $\mathcal{A}s_1^{(-)}\<X\>$ satisfies the following identity of degree $5$:
\begin{multline*}
[[[[a, b], c], d], e]-[[[[a, b], c], e], d]+[[[[a, b], d], c], e]-[[[[a,b], d], e], c]+[[[[a, b], e], c], d]\\
+[[[[a, b], e], d], c]-[[[[a, c], b], d], e]+[[[[a, c], d], b], e]-[[[[a, d], b], c], e]+[[[[a, d], c], b], e]=0.
\end{multline*}    
\end{theorem}
\begin{proof}
The result can be proved using computer algebra systems such as Albert \cite{Albert}.
\end{proof}

\begin{theorem}
All identities of the algebra $\mathcal{A}s_1^{(+)}\<X\>$ up to degree $4$ follow from from commutativity and the following identities:
\[
\{a,\{b,c\}\}+\{b,\{c,a\}\}+\{c,\{a,b\}\}=0
\]
and
\begin{multline*}
\{\{\{a,b\},c\},d\}+\{\{\{a,b\},d\},c\}+\{\{\{a,c\}, b\},d\}+\{\{\{a,c\},d\},b\}\\
+\{\{\{a,d\},b\},c\}+\{\{\{a,d\},c\},b\}=0.
\end{multline*}
\end{theorem}
\begin{proof}
Let us consider the free mock-Lie algebra denoted by $\Mock\Lie\<X\>$, i.e., it is a commutative algebra that satisfies the above identities. The basis monomials of the multilinear component of $\Mock\Lie\<X\>$ in degrees 3 and 4 are
\[
\{\{a,b\},c\},\;\{\{a,c\},b\}
\]
and 
\[
\{\{\{a,b\},d\},c\},\;\{\{\{a,c\}, b\},d\},\;\{\{\{a,c\},d\},b\},\;\{\{\{a,d\},b\},c\},\;\{\{\{a,d\},c\},b\},
\]
respectively. To prove that there are no additional identities in the algebra $\mathcal{A}s_1^{(+)}\<X\>$ up to degree 4, it is sufficient to show that the following equations admit only the trivial solutions:
\[
\lambda_1\{\{a,b\},c\}+\lambda_2\{\{a,c\},b\}=0,
\]
\[
\beta_1\{\{\{a,b\},d\},c\}+\beta_2\{\{\{a,c\},b\},d\}+\beta_3\{\{\{a,c\},d\},b\}+\beta_4\{\{\{a,d\},b\},c\}+\beta_5\{\{\{a,d\},c\},b\}=0.
\]
The calculations for these equations are straightforward.

\end{proof}

\begin{theorem}
An algebra $\mathcal{A}s_1^{(+)}\<X\>$ satisfies the following identity of degree $5$:
\[
\{\{\{\{a,b\},c\},d\},e\}+\{\{\{\{a,b\},d\},c\},e\}=0.
\]
\end{theorem}
\begin{proof}
The result can be proved using computer algebra systems such as Albert \cite{Albert}.
\end{proof}

\section{Second-type associative algebras}

\begin{definition}
An algebra is called a second-type associative algebra if it satisfies the identity:
$$abc-bac-acb+cab+bca-cba=0.$$
\end{definition}

Let us denote by $\mathcal{A}s_2$ the variety of second-type associative algebras. One motivation for considering algebras in $\mathcal{A}s_2$ is the following result:
\begin{proposition}
\cite{assos1} The dual operad of $\mathcal{A}s_2$ is the assosymmetric operad.
\end{proposition}
In \cite{assos1}, it was proved that an operad $\mathcal{A}s_2$ is not Koszul.

Let us list the defining identities of the operad $\Der\mathcal{A}s_2:=\Nov\circ\mathcal{A}s_2$. Direct calculations give a noncommutative Novikov operad with the following identities:
\begin{multline*}
(a\succ b)\succ c-a\succ (b\succ c)-(a\succ c)\succ b+a\succ (c\succ b)-(b\prec a)\succ c+b\succ (a\succ c)\\
+b\prec(c\prec a)-(b\prec c)\prec a+(c\prec a)\succ b-c\succ (a\succ b)-c\prec(b\prec a)+(c\prec b)\prec a=0
\end{multline*}
and
\[
a\succ(b\succ c)-(a\succ c)\prec b-b\succ(a\succ c)+(b\succ c)\prec a+(c\prec a)\prec b-(c\prec b)\prec a=0.
\]

\begin{proposition}
Any algebra in the variety $\Der\mathcal{A}s_2$ can be embedded into an appropriate algebra from the variety $\mathcal{A}s_2^{(d)}$.
\end{proposition}
\begin{proof}
Since $\Nov \circ \mathcal{A}s_2 = \Nov \otimes \mathcal{A}s_2$, the result follows directly from \cite{KolMashSar} and \cite{erlagol2021}.
\end{proof}

Let us compute the black Manin product $\pre\text{-}\mathcal{A}s_2:=\pre\text{-}\Lie \bullet \mathcal{A}s_2$, and we obtain a dendriform operad with the following identity:
\begin{multline*}
(a\preceq b)\preceq c-(b\succeq a)\preceq c-(a\preceq c)\preceq b+(c\succeq a)\preceq b\\
+(b\succeq c)\succeq a+(b\preceq c)\succeq a-(c\succeq b)\succeq a-(c\preceq b)\succeq a=0.
\end{multline*}

\begin{proposition}
Let us define a new operation $\star$ on operad $\pre\text{-}\mathcal{A}s_2$ as follows:
\[a\star b=a\preceq b+a\succeq b.\]
Then an algebra $(X,\star)$ is $\mathcal{A}s_2$. Moreover, the free algebra $\pre\text{-}\mathcal{A}s_2\<X\>$ can be embedded into the algebra $\mathcal{A}s_2\<X\>$ equipped with a Rota–Baxter operator $R$ as follows:
\[
a\preceq b=aR(b)\;\;\textrm{and}\;\;a\succeq b=R(a)b.
\]
\end{proposition}

\begin{proposition}
The polarization of $\mathcal{A}s_2$ is given by the identities:
\[
[\{a,c\},b]=\{[a,b],c\}-\{a,[b,c]\},
\]
\[
[b,[a,c]]=\{\{a,b\},c\}-\{a,\{b,c\}\}
\]
and
\[
\{[a,b],c\}+\{[b,c],a\}+\{[c,a],b\}=0.
\]
\end{proposition}
\begin{proof}
The result can be proved in the same manner as in Proposition \ref{polarization}. The main difference lies in the fact that, starting from the identity
\begin{multline*}
abc-bac-acb+cab+bca-cba=2[[a,b],c]+2\{[a,b],c\}-2[[a,c],b]\\
-2\{[a,c],b\}+2[[b,c],a]+2\{[b,c],a\},
\end{multline*}
the last row of the matrix $[R]$ is replaced by the vector
$$(1/2,-1/2,1/2,0,0,0,1/2,-1/2,1/2).$$
\end{proof}

\begin{theorem}
All identities of the algebra $\mathcal{A}s_2^{(-)}\<X\>$ up to degree $4$ follow from anti-commutativity, the Jacobi identity, and the metabelian identity, which is
\[
[[[a,b],[c,d]]=0.
\]
\end{theorem}
\begin{proof}
Verification of the listed identities can be performed using computer algebra systems such as Albert \cite{Albert}.

Let us denote by $\M\Lie\<X\>$ a free metabelian Lie algebra. The basis monomials of the multilinear component of $\M\Lie\<X\>$ in degrees 3 and 4 are
\[
[[a,b],c],\;[[a,c],b]
\]
and 
\[
[[[b,a],c],d],\;[[[c,a],b],d],\;[[[d,a],b],c],
\]
respectively. To prove that there are no additional identities in the algebra $\mathcal{A}s_2^{(-)}\<X\>$ up to degree 4, it is sufficient to show that the following equations admit only the trivial solutions:
\[
\lambda_1[[a,b],c]+\lambda_2[[a,c],b]=0,
\]
\[
\beta_1[[[a,b],c],d]+\beta_2[[[a,c],b],d]+\beta_3[[[a,d],b],c]=0.
\]
The calculations for these equations are straightforward.

\end{proof}

\begin{theorem}
All identities of algebra $\mathcal{A}s_2^{(+)}\<X\>$ up to degree $4$ follow from commutativity and the following identities:
\begin{multline*}
\{\{\{a,b\},c\},d\}+\{\{\{b,d\},c\},a\}+\{\{\{d,a\},c\},b\}\\
-\{\{a,b\},\{c,d\}\}-\{\{a,c\},\{b,d\}\}-\{\{a,d\},\{b,c\}\}=0,
\end{multline*}
\begin{multline*}
\{\{\{a,b\},c\},d\}-\{\{\{a,b\},d\},c\}-\{\{\{a,c\}, b\},d\}+\{\{\{a,c\},d\},b\}\\
+\{\{\{a,d\},b\},c\}-\{\{\{a,d\},c\},b\}=0.
\end{multline*}
\end{theorem}
\begin{proof}
Verification of the listed identities can be performed using computer algebra systems such as Albert \cite{Albert}.

Let us denote by $\mathcal{S}\mathcal{A}s_2^{(+)}\<X\>$ a free Jordan algebra with additional identity
\begin{multline*}
\{\{\{a,b\},c\},d\}-\{\{\{a,b\},d\},c\}-\{\{\{a,c\}, b\},d\}+\{\{\{a,c\},d\},b\}\\
+\{\{\{a,d\},b\},c\}-\{\{\{a,d\},c\},b\}=0.
\end{multline*}
The basis monomials of the multilinear component of $\mathcal{S}\mathcal{A}s_2^{(+)}\<X\>$ in degrees 3 and 4 are
\[
\{\{a,b\},c\},\;\{\{b,c\},a\},\{\{a,c\},b\}
\]
and 
\[
\{\{\{a,b\},c\},d\},\;\{\{\{a,b\},d\},c\},\;\{\{\{a,c\},b\},d\},\;\{\{\{a,c\},d\},b\},\;\{\{\{a,d\},c\},b\},
\]
\[
\{\{\{a,b\},\{c,d\}\},\;\{\{\{a,c\},\{b,d\}\},\;\{\{\{a,d\},\{b,c\}\},\;\{\{a,\{b,d\}\},c\},\;\{\{a,\{b,c\}\},d\},
\]
respectively. To prove that there are no additional identities in the algebra $\mathcal{A}s_2^{(+)}\<X\>$ up to degree 4, it is sufficient to show that the following equations admit only the trivial solutions:
\[
\lambda_1\{\{a,b\},c\}+\lambda_2\{\{b,c\},a\}+\lambda_3\{\{a,c\},b\}=0,
\]
\begin{multline*}
\beta_1\{\{\{a,b\},c\},d\}+\beta_2\{\{\{a,b\},d\},c\}+\beta_3\{\{\{a,c\},b\},d\}+\beta_4\{\{\{a,c\},d\},b\}\\+\beta_5\{\{\{a,d\},c\},b\}
+\beta_6\{\{\{a,b\},\{c,d\}\}+\beta_7\{\{\{a,c\},\{b,d\}\}+\beta_8\{\{\{a,d\},\{b,c\}\}\\
+\beta_9\{\{a,\{b,d\}\},c\}+\beta_10\{\{a,\{b,c\}\},d\}=0. 
\end{multline*}
The calculations for these equations are straightforward.

\end{proof}

\section{Third-type associative algebras}

\begin{definition}
An algebra is called a third-type associative algebra if it satisfies the following identity:
$$abc+bac-bca-cba=0.$$
\end{definition}

Let us denote by $\mathcal{A}s_3$ the variety of third-type associative algebras.

\begin{proposition}
The dual operad of $\mathcal{A}s_3$ is a left-alternative operad with identity
\begin{equation}\label{sym}
(x,y,z)+(y,z,x)+(z,x,y)=0,    
\end{equation}
where $(x,y,z)$ stands for associator.
\end{proposition}
In \cite{-1-1}, left-alternative algebra with such an identity is called $(-1,1)$-algebra.
\begin{proof}
Firstly, let us fix a multilinear basis of algebra $\mathcal{A}s_3$ of degree $3$. These monomials are $abc$, $acb$, $bac$ and $bca$, and the remaining monomials can be written as a linear combination as follows:  
$$cab=bac+abc-acb$$
and
$$cba=abc+bac-bca.$$

The Lie-admissibility condition for $S\otimes U$ gives the defining identities of the operad $\mathcal{A}s_3^!$, where $S$ is a third-type associative algebra.
Then
\begin{multline*}
[[a\otimes u,b\otimes v],c\otimes w]=(ab)c\otimes (uv)w-(ba)c\otimes (vu)w-
c(ab)\otimes w(uv)+c(ba)\otimes w(vu)=\\
abc\otimes (uv)w-bac\otimes (vu)w-
(bac+abc-acb)\otimes w(uv)+(abc+bac-bca)\otimes w(vu),
\end{multline*}
\begin{multline*}
[[b\otimes v,c\otimes w],a\otimes u]=(bc)a\otimes (vw)u-(cb)a\otimes (wv)u-
a(bc)\otimes u(vw)+a(cb)\otimes u(wv)=\\
bca\otimes (vw)u-(abc+bac-bca)\otimes (wv)u-
abc\otimes u(vw)+acb\otimes u(wv)
\end{multline*}
and
\begin{multline*}
[[c\otimes w,a\otimes u],b\otimes v]=(ca)b\otimes (wu)v-(ac)b\otimes (uw)v-
b(ca)\otimes v(wu)+b(ac)\otimes v(uw)=\\
(bac+abc-acb)\otimes (wu)v-acb\otimes (uw)v-
bca\otimes v(wu)+bac\otimes v(uw).
\end{multline*}
Calculating the sum and collecting the same basis monomials, we obtain
\begin{multline*}
[[a\otimes u,b\otimes v],c\otimes w]+[[b\otimes v,c\otimes w],a\otimes u]+[[c\otimes w,a\otimes u],b\otimes v]=\\
abc\otimes((uv)w-w(uv)+w(vu)-u(vw)-(wv)u+(wu)v)\\
+acb\otimes(w(uv)+u(wv)-(uw)v-(wu)v)\\
+bac\otimes(-(vu)w+w(vu)-w(uv)+v(uw)+(wu)v-(wv)u)\\
+bca\otimes(-w(vu)+(vw)u+(wv)u-v(wu))=0.
\end{multline*}
From the right sides of the tensors, we obtain the result.
\end{proof}

\begin{proposition}\label{notKoszul}
The operad $\mathcal{A}s_3$ is not Koszul.
\end{proposition}
\begin{proof}
Firstly, let us calculate the dimensions of the operads $\mathcal{A}s_3$ and $\mathcal{A}s_3^{\;!}$ by means of the package \cite{DotsHij}, 
we get the following results:
\begin{center}
\begin{tabular}{c|ccccccc}
 $n$ & 1 & 2 & 3 & 4 & 5 \\
 \hline 
 $\dim(\mathcal{A}s_3(n)) $ & 1 & 2 & 4 & 1 & 1
\end{tabular}
\end{center}
and
\begin{center}
\begin{tabular}{c|ccccccc}
 $n$ & 1 & 2 & 3 & 4 & 5 \\
 \hline 
 $\dim(\mathcal{A}s_3^{\;!}(n)) $ & 1 & 2 & 8 & 41 & 213
\end{tabular}
\end{center}
According to the obtained tables, the first few terms of the Hilbert series of the operads $\mathcal{A}s_3$ and $\mathcal{A}s_3^{\;!}$ are
$$H(t)=-t+t^2-4t^3/6+t^4/24-t^5/120+O(t^6)$$
and
$$H^!(t)=-t+t^2-8t^3/6+41t^4/24-213t^5/120+O(t^6)$$
Thus,
$$H(H^!(t))=t+t^5/60+O(t^6)\neq t.$$
By \cite{GK94}, the operad $\mathcal{P}_2^{\;\;!}$ is not Koszul.
\end{proof}

Let us list the defining identities of the operad $\Der\mathcal{A}s_3:=\Nov\circ\mathcal{A}s_3$. Straightforward calculations give a noncommutative Novikov operad with the following identities:
\begin{multline*}
(a\succ b)\succ c-a\succ (b\succ c)+b\prec (a\succ c)-(b\prec a)\prec c\\
-b\prec(c\prec a)+(b\prec c)\prec a-c\prec (b\prec a)+(c\prec b)\prec a=0,
\end{multline*}
\begin{multline*}
(b\prec a)\succ c-b\succ (a\succ c)+(a\succ b)\succ c-a\succ (b\succ c)\\
-(a\succ c)\succ b+a\succ (c\succ b)-(c\prec a)\succ b+c\succ(a\succ b)=0,
\end{multline*}
\[
a\succ(b\succ c)+b\succ(a\succ c)-(b\succ c)\prec a-(c\prec b)\prec a=0,
\]
\[
(a\prec b)\prec c+(b\succ a)\prec c-b\succ(c\succ a)-c\succ(b\succ a)=0
\]
and
\[
(a\succ b)\prec c+(b\prec a)\prec c-(b\prec c)\prec a-(c\succ b)\prec a=0.
\]

\begin{proposition}
Any algebra in the variety $\Der\mathcal{A}s_3$ can be embedded into an appropriate algebra from the variety $\mathcal{A}s_3^{(d)}$.
\end{proposition}
\begin{proof}
Since $\Nov\circ\mathcal{A}s_3=\Nov\otimes\mathcal{A}s_3$, the result immediately follows from \cite{KolMashSar} and \cite{erlagol2021}.
\end{proof}

Let us compute $\pre\textrm{-}\mathcal{A}s_3:=\pre\textrm{-}\Lie\bullet\mathcal{A}s_3$, and we obtain a dendriform operad with the following identities:
\[
a\succeq(b\succeq c)+b\succeq(a\succeq c)=(b\succeq c)\preceq a+(c\preceq b)\preceq a,
\]
\[
a\succeq(b\preceq c)+(b\preceq a)\preceq c=(b\preceq c)\preceq a+c\succeq(b\preceq a)
\]
and
\[
(a\preceq b)\preceq c+(b\succeq a)\preceq c=b\succeq(c\succeq a)+c\succeq(b\succeq a).
\]

\begin{proposition}
Let us define a new operation $\star$ on operad $\pre\textrm{-}\mathcal{A}s_3$ as follows:
\[a\star b=a\preceq b+a\succeq b.\]
Then an algebra $(X,\star)$ is $\mathcal{A}s_3$. Indeed, a free algebra $\pre\textrm{-}\mathcal{A}s_3\<X\>$ can be embedded into the algebra $\mathcal{A}s_3\<X\>$ with the Rota-Baxter operator $R$ as follows:
\[
a\preceq b=aR(b)\;\;\textrm{and}\;\;a\succeq b=R(a)b.
\]
\end{proposition}

\begin{proposition}
The polarization of $\mathcal{A}s_3$ is
\[
[\{a,c\},b]=\{[a,b],c\}-\{a,[b,c]\},
\]
\[
[b,[a,c]]=\{\{a,b\},c\}-\{a,\{b,c\}\},
\]
\[
\{a,\{b,c\}\}-3/2\{[a,c],b\}-3/2\{[a,b],c\}-1/2\{\{a,c\},b\}-1/2\{\{a,b\},c\}=0,
\]
and
\[
\{a,[b,c]\}= -1/2\{[a,c],b\}+1/2\{[a,b],c\}+ 1/2\{\{a,c\},b\}-1/2\{\{a,b\},c\}.
\]
\end{proposition}
\begin{proof}
The result can be proved similarly to Proposition \ref{polarization}. Consider the following identities:
$$abc+bac-bca-cba=2[\{a,b\},c]+2\{\{a,b\},c\}-2[\{b,c\},a]-2\{\{b,c\},a\}$$
and
\begin{multline*}
abc+acb-cba-cab=[[a,b],c]+[\{a,b\},c]+\{[a,b],c\}+\{\{a,b\},c\}\\
+2[[a,c],b]+2\{[a,c],b\}+[[b,c],a]-[\{b,c\},a]+\{[b,c],a\}-\{\{b,c\},a\}.
\end{multline*}
These identities correspond to the following row vectors:
$$(0,0,0,-1/2,0,1/2,0,0,0,-1/2,0,1/2)$$
and
$$(1/4,1/2,1/4,-1/4,0,1/4,1/4,1/2,1/4,-1/4,0,1/4),$$
respectively. To obtain the desired result, we replace the last row of the matrix $[R]$ with these two vectors.
\end{proof}

\begin{theorem}\label{commutatorAs3}
All identities of algebra $\mathcal{A}s_3^{(-)}\<X\>$ up to degree $4$ follow from anti-commutative, Jacobi and the following identity:
\[
[[[a,b],c],d]=0.
\]
\end{theorem}
\begin{proof}
Verification of the listed identities can be performed using computer algebra systems such as Albert \cite{Albert}.

Since every monomial in the free Lie algebra can be expressed as a linear combination of left-normed monomials, it suffices to show that there are no new identities in degree $3$. This is equivalent to showing that the following equation admits only the trivial solution:
\[
\lambda_1[[a,b],c]+\lambda_2[[a,c],b]=0.
\]
The calculations of this equation are straightforward.
\end{proof}

\begin{theorem}\label{anti-commutatorAs3}
All identities of algebra $\mathcal{A}s_3^{(+)}\<X\>$ up to degree $4$ follow from commutative and the following identities:
\[
\{\{\{a,b\},c\},d\}=\{\{\{a,c\},b\},d\},
\]
and
\[
\{\{\{a,b\},c\},d\}=\{\{\{a,b\},d\},c\}=\{\{\{a,b\},\{c,d\}\}.
\]
\end{theorem}
\begin{proof}
Verification of the listed identities can be performed using computer algebra systems such as Albert \cite{Albert}.

Let us consider the free commutative algebra satisfying the identities
\[
\{\{\{a,b\},c\},d\}=\{\{\{a,c\},b\},d\}=\{\{\{a,b\},d\},c\}=\{\{\{a,b\},\{c,d\}\},
\]
which we denote by $\mathcal{S}\mathcal{A}s_3^{(+)}\<X\>$. The basis monomials of the multilinear component of $\mathcal{S}\mathcal{A}s_3^{(+)}\<X\>$ in degrees 3 and 4 are
\[
\{\{a,b\},c\},\;\{\{b,c\},a\},\{\{a,c\},b\}
\]
and 
\[
\{\{\{a,b\},c\},d\},
\]
respectively. To prove that there are no additional identities in the algebra $\mathcal{A}s_3^{(+)}\<X\>$ up to degree 4, it is sufficient to show that the following equations admit only the trivial solutions:
\[
\lambda_1\{\{a,b\},c\}+\lambda_2\;\{\{b,c\},a\}+\lambda_3\{\{a,c\},b\}=0,
\]
\[
\beta_1\{\{\{a,b\},c\},d\}=0.
\]
The calculations for these equations are straightforward.
\end{proof}

\begin{remark}
There is no need to consider the variety of associative algebras satisfying identity (\ref{as4}), i.e. the variety $\mathcal{A}s_4$, as the results for the commutator and anti-commutator in this case coincide with Theorems \ref{commutatorAs3} and \ref{anti-commutatorAs3}, respectively. Moreover, an analogue of Proposition \ref{notKoszul} holds for the operad $\mathcal{A}s_4$ due 
an algebra $\mathcal{A}s_3\<X\>$ isomorphic to $\mathcal{A}s_4\<X\>$ under the mapping defined by opposite operation.
\end{remark}

Let us summarise some of the obtained results in the following table:
\begin{table}[h!]
\centering
\begin{tabular}{|c|p{3.5cm}|p{3cm}|c|p{3cm}|p{3cm}|}
\hline
\textbf{Type} & \textbf{Defining Identity} & \textbf{Dual Operad} & \textbf{Koszul} & \textbf{Commutator} & \textbf{Anti-commutator} \\
\hline
I & $abc + bac + acb + cab + bca + cba = 0$ & Alternative & No & Lie algebra with additional identity of degree $5$ & Mock-Lie algebra with additional identity of degree $5$ \\
\hline
II & $abc - bac - acb + cab + bca - cba = 0$ & Assosymmetric & No & Metabelian Lie algebra & Jordan algebra with additional identity of degree $4$ \\
\hline
III & $abc + bac - bca - cba = 0$ & Left-alternative and additional identity (\ref{sym}) & No & Nilpotent Lie algebra & Jordan algebra with additional identity of degree $4$ \\
\hline
IV & $abc + acb - cba - cab = 0$ & Right-alternative and additional identity (\ref{sym}) & No & Nilpotent Lie algebra & Jordan algebra with additional identity of degree $4$ \\
\hline
\end{tabular}
\end{table}

Open problems:
\begin{enumerate}
    \item Provide a complete list of identities satisfied by the first-type associative algebra under the commutator and anti-commutator.
    \item Is it possible to embed every metabelian Lie algebra into an appropriate second-type associative algebra via the commutator?
    \item Describe the class of Jordan algebras that can be embedded into a suitable second-type associative algebra via the anti-commutator.

\end{enumerate}

\subsection*{Author Contributions}
The first author did all calculations and revised the main text. The second author suggested the main idea and wrote the main text.

\subsection*{Acknowledgments}
This research was funded by the Science Committee of the Ministry of Science and Higher Education of the Republic of Kazakhstan (Grant No. AP23484665).

\end{document}